\documentclass[reqno,12pt]{amsart}
\usepackage{amsmath, amssymb, amsthm, epsfig}

\usepackage{color}
\usepackage[inner=1.0in,outer=1.0in,bottom=1.0in, top=1.0in]{geometry}

\usepackage{hyperref, latexsym}
\usepackage{url}

\usepackage{color}

\def\today{\ifcase\month\or
  January\or February\or March\or April\or May\or June\or
  July\or August\or September\or October\or November\or December\fi
  \space\number\day, \number\year}

 \newtheorem{theorem}{Theorem}[section]
\newtheorem{conj}[theorem]{Conjecture}
 \newtheorem{lemma}[theorem]{Lemma}
 \newtheorem{prop}[theorem]{Proposition}
 
 \theoremstyle{definition}

 \theoremstyle{remark}

\newcommand{\sumstar}{\sideset{}{^*}\sum}

\def\L{\mathbf L}
\def\B{\mathbf B}

\newcommand{\vecx}{\boldsymbol{x}}
\newcommand{\vecy}{\boldsymbol{y}}

\newcommand{\vecu}{\boldsymbol{u}}
\newcommand{\vecv}{\boldsymbol{v}}

\newcommand{\vecz}{\boldsymbol{z}}
\newcommand{\veck}{\boldsymbol{k}}
\newcommand{\vecl}{\boldsymbol{\ell}}
\newcommand{\vecm}{\boldsymbol{m}}
\newcommand{\vecn}{\boldsymbol{n}}

\begin{document}

\title{On the zeros of Epstein zeta functions near the critical line} 

\date{\today}

\author[Y. Lee]{Yoonbok Lee}

\allowdisplaybreaks
\numberwithin{equation}{section}

\maketitle

\begin{abstract}
Let $Q$ be a positive definite quadratic form with integral coefficients and let $E(s,Q)$ be the Epstein zeta function associated with $Q$. Assume that the class number of $Q$ is bigger than $1$. Then we estimate the number of zeros of $E(s,Q)$ in the region $   \Re s > \sigma_T ( \theta  ) := 1/2 + ( \log T)^{- \theta}$ and $ T < \Im s < 2T$, to provide its asymptotic formula for fixed $ 0 < \theta < 1$ conditionally. Moreover, it is unconditional if the class number of $Q$ is $2$ or $3$ and $  0 < \theta  < 1/13$.
\end{abstract}

\section{Introduction}

Let $ K = \mathbb{Q}( \sqrt{D}) $ be a quadratic imaginary field of class number $h:=h_D $ and let $ \chi_1, \dots, \chi_h$ be its ideal class characters. The Hecke $L$-function attached to $\chi_j$ is defined by
 $$ L_j ( s) := L(s, \chi_j) := \sum_{\mathfrak{n}}   \frac{ \chi_j(\mathfrak{n})}{   \mathcal{N}(\mathfrak{n}) ^s } = \prod_{\mathfrak{p}}  \bigg(  1- \frac{ \chi_j( \mathfrak{p})}{  \mathcal{N}(\mathfrak{p})^s } \bigg)^{-1}  $$ 
 for $ \Re s > 1 $, where $ \mathcal{N}$ is the norm. 
Each $L_j$ has an analytic continuation to $\mathbb{C}$ except for a possible pole at $s=1$ and it satisfies the functional equation 
 \begin{equation}\label{func eqn 1}
  \bigg(  \frac{ \sqrt{-D}}{ 2 \pi} \bigg)^s \Gamma(s) L_j (s) = \bigg(  \frac{ \sqrt{-D}}{ 2 \pi} \bigg)^{1-s} \Gamma(1 - s) L_j (1-s)   .
  \end{equation}
  By the Euler product and \eqref{func eqn 1}, $L_j$ has no zeros in $\Re s > 1 $ and the negative integers are the only zeros of $L_j $ in $\Re s < 0$. All the other zeros are on the strip $0\leq \Re s \leq 1 $ and we believe that they are actually on the line $\Re s = 1/2$.  
 
  These Hecke $L$-functions have a functional relation with an Epstein zeta function. To be precise, let $Q$ be a positive definite quadratic form with integral coefficients and its fundamental discriminant $D$.  The Epstein zeta function $E(s,Q)$ associated with $Q$ is defined by
 $$
 E(s,Q): = \sum_{\substack{ m,n \in \mathbb{Z} \\ (m,n) \neq (0,0) }}  \frac{1}{   Q(m,n)^s} 
 $$
 for $ \Re s > 1 $. It satisfies
  \begin{equation}\label{func rel}
 E(s,Q) = \frac{w_D}{ h_D } \sum_j  \overline{  \chi_j ( \mathfrak{a}_Q )} L_j (s) ,
 \end{equation}
 where $w_D$ is the number of roots of unity in $K$ and $ \mathfrak{a}_Q $ is an integer ideal in the ideal class corresponding to the equivalence class of $Q$. If $h_D=1$, then the Epstein zeta function $E(s,Q) = w_D L_1 (s)$ is nothing but a Hecke $L$-function up to a constant factor. Hence we expect $E(s,Q)$ satisfy the generalized Riemann hypothesis. 
 However, if $h_D>1$, the distribution of zeros of $E(s,Q)$ is different to the Riemann zeta function and indeed $E(s,Q)$ has zeros off the line $\Re s = 1/2$. Davenport and Heilbronn \cite{DH} showed that $E(s,Q)$ has infinitely many zeros on $\Re s > 1$. Voronin \cite{V} showed that the number of zeros of $E(s,Q)$ in the rectangle $ \sigma_1 < \Re s < \sigma_2 $ and $ T < \Im s < 2T$ is
 $$ N_{E(s,Q)} (\sigma_1, \sigma_2 : T , 2T)  \gg T$$
 for any fixed $ 1/2 < \sigma_1 < \sigma_2 < 1 $ as a consequence of a joint universality of Hecke $L$-functions. The author in \cite{Le2} proved that
  $$ \lim_{T \to \infty } \frac1T  N_{E(s,Q)} (\sigma_1, \sigma_2 : T , 2T)  =    \int_{\sigma_1}^{\sigma_2} g(\sigma) d\sigma  $$
holds for any $1/2 < \sigma_1 \leq \sigma_2 $ and some nonnegative continuous function $g(\sigma)$. 
By a straightforward adaptation of \cite{LLR}, the author in \cite{Le3} improved the above asymptotic formula to
$$\frac1T   N_{E(s,Q)} (\sigma_1, \sigma_2 : T , 2T)  =      \int_{\sigma_1}^{\sigma_2} g(\sigma) d\sigma   + O\bigg(    \frac{ \log \log T} { ( \log T)^{ \sigma_1/2}}  \bigg)  $$
for $ h = 2,3$ and fixed $ 1/2 < \sigma_1 < \sigma_2 <1 $. The author with Gonek in \cite{GL} considered the case $ h>3$ and proved that 
$$\frac1T N_{E(s,Q)} (\sigma_1, \sigma_2 : T , 2T)  =     \int_{\sigma_1}^{\sigma_2} g(\sigma) d\sigma    + O(  e^{- b \sqrt{\log \log T}} )  $$
for fixed $ 1/2 < \sigma_1 < \sigma_2 <1 $.

   Now we examine the zero-density on or near the $1/2$-line for a linear combination of $L$-functions with a same functional equation, which generalizes our Epstein zeta functions. Bombieri and Hejhal in \cite{BH} proved that almost all zeros of a linear combination $F(s)$ of inequivalent $L$-functions   with a same functional equation are simple and on the $1/2$-line  assuming RH and a zero-spacing assumption for each $L$-function. Hejhal in \cite{H}  investigated the zeros of $F (s) = (\cos \alpha) e^{i w_1} L_1 (s) + (\sin \alpha) e^{iw_2}  L_2 (s) $ near the $1/2$-line and showed that  
 \begin{equation}\label{zero density Hejhal}
 \frac{ T \log T}{  G \sqrt{ \log \log T}}\ll     N_{F(s) } ( \sigma_1, \infty : T, 2T) \ll\frac{ T \log T}{  G \sqrt{ \log \log T}} 
 \end{equation}
 for $ \sigma_1 = 1/2 + G/ \log T $ and for almost all $\alpha$, where $ (\log \log T)^\kappa  \leq G \leq (\log T)^{1-\delta}$, $ \kappa \in (1,3)$ and $ \delta \in (0, 1/10)$. We expect that it holds for all $\alpha$ except for the cases $ \cos \alpha =0 $ and $ \sin \alpha = 0$, but it is still an open question whether a given Epstein zeta function with class number 2 or 3 satisfy \eqref{zero density Hejhal}. Selberg in \cite{S} sketched his idea which proves \eqref{zero density Hejhal} for almost all linear combinations $F(s)$ of $L$-functions.

 The aim of this paper is finding an asymptotic formula for the zero counting function
$$   N_{E(s,Q)} ( \sigma_T, \infty : T, 2T)$$
of a given Epstein zeta function $E(s,Q)$ as $T\to \infty$, where  
$$ \sigma_T = \sigma_T(\theta) = \frac12 + \frac{1}{( \log T)^\theta}$$
for fixed $ 0 < \theta < 1 $. Let $L_1, \ldots, L_J$ be inequivalent Hecke $L$-functions on $K$, i.e., $L_j (s) = L(s, \chi_j )$ and $ \chi_j \neq \chi_\ell, \overline{\chi}_\ell$ for $ j \neq \ell $. By the Euler product, we may write
$$ \log L( s, \chi_j ) = \sum_p \sum_{n=1}^\infty    \frac{ a_j ( p^n)}{ p^{ns}}$$
for $\Re s > 1 $, then it is well-known that
\begin{equation}\label{eqn:soc} 
\sum_p \frac{ a_j (p) a_\ell(p) }{ p^{2 \sigma_T}}   = \delta_{j,\ell}  \xi_j \theta \log \log T + c_{j,\ell} + O\bigg( \frac{1}{ ( \log T)^\theta} \bigg)
\end{equation} 
for $ j ,\ell \leq J $, where $\delta_{j,\ell} = 1 $ if $ j = \ell$, $\delta_{j,\ell}=0$ if $ j \neq \ell $, $\xi_j =4 $ if $\chi_j $ is real, $ \xi_j = 2 $ if $\chi_j $ is nonreal and the $c_{j, \ell}$ are some constants. Consider
 $$ F_J (s) := \sum_{j=1}^J b_j L_j (s)$$
 for $b_1 , \ldots, b_J \in \mathbb{C}\setminus \{ 0 \} $ satisfying
 $$|b_1|^2 + \cdots + |b_J|^2 = 1 . $$
 The Epstein zeta function $E(s,Q)$ is a special case of $F_J (s)$ up to a constant factor by \eqref{func rel} and the relation $L(s,\chi ) = L(s,\bar{\chi})$. In this case, $J$ is the sum of the number of real characters and the half of the number of non-real characters. Hence, $J=2$ if $h=2,3$ and $ J > 2 $ if $h > 3 $.

By Littlewood's lemma, the zero counting function
$$ N_{F_J} ( \sigma : T) :=  N_{ F_J } ( \sigma , \infty :  T, 2T) $$ 
for $ \sigma > 1/2$ is essentially a derivative of the integral
$$ \frac{1}{ 2 \pi}  \int_T^{2T} \log | F_J ( \sigma  +it) | dt . $$
Moreover, we believe that the following conjecture is true.
\begin{conj}\label{conj}
Let $ J >1 $, $ 0 < \theta < 1 $ and $ \sigma_T = 1/2+ ( \log T)^{-\theta}$. Then there exists $ \eta>0$ such that
\begin{equation}
 \frac{1}{ T }  \int_T^{2T} \log | F_J ( \sigma_T   +it) | dt  =      \mathbb{E} \big[  \log | F_J ( \sigma_T  :  X ) |  \big]  + O \bigg(  \frac{1}{ ( \log T)^\eta} \bigg)
 \end{equation}
as $ T \to \infty$, where
$$ F_J ( s : X ) :=  \sum_{j=1}^J b_j  L_j ( s: X) $$
and
$$ L_j ( s: X) := \prod_{\mathfrak{p}} \bigg(  1 -    \frac{ \chi_j ( \mathfrak{p}) X(  \mathcal{N}(\mathfrak{p}) )}{ \mathcal{N}(\mathfrak{p})^s }  \bigg)^{-1}$$
 are the random models of $F_J (s)$ and $ L_j(s)$ for $j=1, \ldots, J$. Here, the $X(p)$ are uniformly and independently distributed on the unit circle $ \mathbb{T}$ and $ X(p^\ell) := X(p)^\ell$.
 \end{conj}
 Conjecture \ref{conj} for $J>2$ is technically more difficult than the estimates in \cite{GL} for fixed $ \sigma > 1/2$, since there are more logarithmic singularties near the $1/2$-line. However, if $J=2$, it is possible to prove Conjecture \ref{conj} for a small $\theta$. One sees that
 $$ \int_T^{2T} \log | F_2 ( \sigma_T + it) | dt =  \int_T^{2T} \log |b_1L_2  ( \sigma_T + it) | dt +  \int_T^{2T} \log\bigg|   \frac{ L_1 }{ L_2  }(\sigma_T + it)  +  \frac{b_2}{b_1} \bigg| dt .$$
The first integral on the right hand side can be estimated by an usual Dirichlet polynomial approximation for $\log |L_2( \sigma_T + it )|$ and the second integral can be estimated by adapting \cite{HL}. Since its proof is straightforward from \cite{HL}, we state it without a proof as follows.
\begin{theorem}\label{thm1}
Conjecture \ref{conj} holds for $J=2$ and $ 0<  \theta < 1/13$ with $\eta < (1-13\theta)/4$. 
\end{theorem}

The main feature of this paper is our estimation of 
 $$ \mathbb{E} \big[  \log | F_J ( \sigma_T   :  X ) |  \big] $$
  as $T \to \infty $ for $J > 1$.
  \begin{theorem}\label{thrm expt} 
Let $ J >1 $, $ 0 < \theta < 1 $ and $ \sigma_T = 1/2+ ( \log T)^{-\theta}$. Then  \begin{align*}
       \mathbb{E} \big[  \log  | F_J ( \sigma_T   :  X ) |  \big] 
  = &   \frac{ \sqrt{  \theta  \log \log T}  }{\sqrt{ \xi \pi^J }}   \sum_{ \ell=1}^J   \int_{\mathcal{R}_\ell }     u_\ell         e^{ -  \sum_j  u_j^2 /  \xi_j  }           d\vecu  \\
 &   +     \sum_{ \ell=1}^J     \frac{   \log |b_\ell |}{\sqrt{ \xi \pi^J }}   \int_{\mathcal{R}_\ell }      e^{ -  \sum_j  u_j^2 /  \xi_j }            d\vecu     +  O\bigg(  \frac{ 1}{  ( \log \log T)^{1/4}}    \bigg) 
  \end{align*}
  as $T \to \infty$, where
$$  \xi :=   \prod_j \xi_j ,$$
  \begin{equation}\label{def Rell}
 \mathcal{R}_{\ell} :=  \{ ( u_1 , \dots, u_J ) \in \mathbb{R}^J  : u_\ell     = \max \{ u_1    , \dots , u_J  \}  \}  
 \end{equation}
for $ \ell = 1, 2,\ldots, J$ and $d \vecu = du_1 \cdots du_J $. 
  \end{theorem}
  One expects that Littlewood's lemma, Conjecture \ref{conj} and Theorem \ref{thrm expt} imply an asymptotic of $N_{F_J} ( \sigma_T :T)  $ as $T \to \infty$, but the $O$-term in Theorem \ref{thrm expt} is too big. Instead, we estimate the difference 
  $$  \mathbb{E} \big[  \log | F_J ( \sigma_T (\theta_1)  :  X ) |  \big] -  \mathbb{E} \big[  \log | F_J ( \sigma_T  (\theta_2)  :  X ) |  \big]  $$
  for a small $ | \theta_1 - \theta_2| $ to prove our main theorem.
\begin{theorem}\label{thrm zero density}
Let $J \geq 2 $ and $ 0 < \theta < 1 $ and assume Conjecture \ref{conj}, then
  \begin{align*}
   N_{F_J} ( \sigma_T( \theta):T)  
=&     \frac{   1    }{4 \pi^{1+J/2}   }      \frac{T ( \log T)^\theta}{  \sqrt{\xi \theta  \log \log T} }       \sum_{ \ell=1}^J   \int_{\mathcal{R}_\ell }     u_\ell         e^{ -  \sum_j  u_j^2 /    \xi_j  }          d\vecu    +O \bigg( \frac{ T( \log T)^\theta}{   ( \log \log T)^{5/4} }  \bigg) 
\end{align*}
as $T \to \infty$, where $\xi$ and $\mathcal{R}_\ell$ are defined in Theorem \ref{thrm expt}.
\end{theorem}  
By Theorem \ref{thm1}, we see that Theorem \ref{thrm zero density} holds for $ J=2$ and $ 0 < \theta < 1/13$ unconditionally. When $F_J(s)$ is the Epstein zeta function $E(s,Q)$ up to a constant factor, it is interesting to see that 
$$\xi = 2^{3J-h} . $$

 \section{Proof of Theorems \ref{thrm expt} and \ref{thrm zero density}}\label{sec proof zero density}
 
   Let
      \begin{align*}
 \L  ( \sigma : X )  := \big(    \log & | L_1  ( \sigma :  X)| , \dots, \log |L_J ( \sigma:X)| ,  \Im \log     L_1( \sigma :  X )   , \dots ,  \Im \log     L_J( \sigma : X )    \big) .  
\end{align*}
 Define
 \begin{equation}\label{Psi}
\Psi_{\theta, T }   (\B ):=
 \mathbb{P}(\L(\sigma_T :  X) \in \B)
 =  \mathrm{meas}  \{ X  \in \mathbb{T}^\infty :  \L(\sigma_T : X) \in \B  \}  
 \end{equation} 
for  a Borel set $  \B$ in $\mathbb R^{2J}$. It is known that the measure $\Psi_{\theta, T}$ has a density function $G_{\theta,T}$, so that for 
 $$ \Psi_{\theta,T} (\B) = \int_\B   G_{\theta,T} (\vecu, \vecv) d\vecu d\vecv, $$
 where $ \vecu = ( u_1 , \dots, u_J ), \vecv = ( v_1 , \dots, v_J )  \in \mathbb{R}^J$. 
(For instance, see the proof of Lemma 3.3 in \cite{GL}.) Then we see that
\begin{equation}\label{eqn int G}\begin{split}
   \mathbb{E} \big[  \log | F_J ( \sigma_T   :  X ) |  \big] & = \int_{ \mathbb{R}^{2J}} \log \bigg|  \sum_{j \leq J} b_j e^{ u_j + i v_j } \bigg|  d \Psi_{\theta, T} ( \vecu, \vecv) \\
   & = \int_{ \mathbb{R}^{2J}} \log \bigg|  \sum_{j \leq J} b_j e^{ u_j + i v_j } \bigg|  G_{\theta, T} (\vecu, \vecv)  d \vecu  d\vecv 
\end{split}   \end{equation}
and
  $$ \hat{\Psi}_{\theta,T} ( \vecx, \vecy) = \int_{ \mathbb{R}^J }e^{2 \pi i ( \vecu \cdot \vecx + \vecv \cdot \vecy )} d\Psi_{\theta, T} (\vecu ,\vecv)=  \int_{ \mathbb{R}^J }e^{2 \pi i ( \vecu \cdot \vecx + \vecv \cdot \vecy )}  G_{\theta,T} (\vecu, \vecv) d\vecu d\vecv ,   $$
  where $ \vecu = ( u_1 , \dots, u_J ), \vecv = ( v_1 , \dots, v_J ), \vecx = ( x_1 , \ldots, x_J), \vecy = ( y_1 , \ldots, y_J) \in \mathbb{R}^J$, $ \vecu \cdot \vecx = u_1 x_1 + \cdots + u_J x_J $ and $ \vecv \cdot \vecy = v_1 y_1 + \cdots + v_J y_J $.
  Since $G_{\theta,T} (\vecu, \vecv) $ is the inverse Fourier transform of $  \hat{\Psi}_{\theta,T} ( \vecx, \vecy) $ 
    $$G_{\theta,T} (\vecu, \vecv)  =  \int_{ \mathbb{R}^J }e^{ - 2 \pi i ( \vecu \cdot \vecx + \vecv \cdot \vecy )}   \hat{\Psi}_{\theta,T} ( \vecx, \vecy)   d\vecu d\vecv   , $$
 we examine various properties of $  \hat{\Psi}_{\theta,T} ( \vecx, \vecy) $ to study $G_{\theta,T} (\vecu, \vecv) $.
 \begin{prop}\label{prop:hat Psi}  There exist constants $c_1, c_2 > 0$ such that
$$   | \hat{\Psi}_{\theta,T} ( \vecx, \vecy) |       \leq \exp \bigg(  - c_1   \frac{ ( \sum_j x_j^2 + y_j^2)^{1/2\sigma_T} }{ \log (\sum_j x_j^2 + y_j^2)    }  \bigg)  $$
for $ \sum_j x_j^2 + y_j^2 \geq c_2 $, 
\begin{align*}
| \hat{\Psi}_{\theta,T} ( \vecx, \vecy) |  & \leq  \exp \bigg(    -   \bigg(  \frac{ \pi^2 \theta}{2}  \log \log T  + O(1) \bigg)  \sum_{j \leq J }  (x_j^2 + y_j^2 ) \bigg)  
\end{align*}
for  $ \sum_{ j \leq J } (x_j^2 + y_j^2 )  \leq   e^{ ( \log T)^{ \theta /2 }}  $, and 
   \begin{align*}
 \hat{\Psi}_{\theta,T} ( \vecx, \vecy)  
= &   e^{   - \pi^2 \theta \log \log T  \sum_{j \leq J} \xi_j (x_j^2 + y_j^2 )}    \bigg( P( \vecx, \vecy)  +  O\bigg(      \sum_{j=1}^J (x_j^2 + y_j^2  )^3  + \frac{1}{ ( \log T)^\theta} \bigg)  \bigg)  
\end{align*}
for $ \sum_j x_j^2 + y_j^2 \leq c_2$, where
$$  P( \vecx, \vecy) :=    \sum_{  \veck , \vecl \in ( \mathbb{Z}_{\geq 0 } )^J }     \tilde{B}_{\veck, \vecl} \vecx^{\veck } \vecy^{\vecl }  $$
and the coefficients $ \tilde{B}_{\veck, \vecl}$ are independent to $\theta$ and $T$ satisfying 
$$ \tilde{B}_{0,0} = 1 $$
and
$$ \tilde{B}_{ \veck, \vecl}=0$$
if $\mathcal{K}(\veck + \vecl) = 1 $ or $ >5 $. Here, $\vecx = ( x_1 , \ldots, x_J), \vecy = ( y_1 , \ldots , y_J ) \in \mathbb{R}^J$, $\veck = ( k_1, \ldots, k_J), \vecl= ( \ell_1 , \ldots , \ell_J) \in (\mathbb{Z}_{\geq 0 })^J$ and 
$$ \mathcal{K}( \veck) := k_1 + \cdots + k_J, \qquad  \vecx^{\veck}:= \prod_{j\leq J} x_j^{k_j} . $$
 \end{prop}
 
 \begin{prop}\label{prop:G}
 Let $\veck, \vecl, \vecm,\vecn$ be vectors in $ (\mathbb{Z}_{\geq 0 })^J$ and 
 $$ q_{ \veck, \vecl: \vecm, \vecn}  = \frac{ \tilde{B}_{ 2\veck+\vecm, 2\vecl+\vecn}   }{   i^{\mathcal{K}( \vecm+ \vecn)} \pi^{2J + \mathcal{K}(2\veck+2\vecl+\vecm+ \vecn) } }\prod_j  \bigg( \frac{ \Gamma (k_j + 1/2)  \Gamma ( \ell_j +1/2)  }{ \xi_j^{k_j+\ell_j + m_j+ n_j+1 }  } \bigg) \frac{ (2 \veck+\vecm )!     (2\vecl+\vecn)!}{ (2\veck)! \vecm! (2\vecl)! \vecn! }   .    $$
  Then, we have 
   \begin{align*}
 G_{\theta, T} ( \vecu, \vecv) =&   \exp\bigg( -  \sum_j \frac{ u_j^2+v_j^2}{ \theta \xi_j \log\log T} \bigg)   \sum_{ \veck, \vecl, \vecm, \vecn \in  (\mathbb{Z}_{\geq 0 })^J    }  \frac{ q_{ \veck, \vecl: \vecm, \vecn}    }{  ( \theta \log \log T)^{ J+ \mathcal{K}(\veck+\vecl+ \vecm+\vecn)  } }\vecu^{\vecm} \vecv^{\vecn}    \\
 &  +  O\bigg(   \frac{1}{ ( \log \log T)^{J+3}}  \bigg) 
 \end{align*}
 for all $ \vecu = ( u_1 , \ldots, u_J) , \vecv = ( v_1 , \ldots, v_J) \in \mathbb{R}^J$, and there exists a constant $c>0$ such that 
 $$ G_{\theta, T} ( \vecu, \vecv) \ll \exp\bigg(   - c\sum_{j \leq J}   \frac{u_j^2 + v_j^2 }{  \log \log T} \bigg) $$
for all $ \vecu = ( u_1 , \ldots, u_J) , \vecv = ( v_1 , \ldots, v_J) \in \mathbb{R}^J$.
 \end{prop}
 Note that  
 $$ q_{0,0:0,0}= \pi^{-J} \prod_j \xi_j^{-1}= \pi^{-J} \xi^{-1}  $$
 and
 $$ q_{ \veck, \vecl: \vecm, \vecn} = 0 $$
 if $ \mathcal{K}(2\veck+2\vecl+ \vecm+\vecn)=1 $ or $ >5$.  We also need the following lemma.
 \begin{lemma}\label{lemma}
 Let $k$ be a positive integer, $M\geq 1$ and $ b_j \in \mathbb{C} $ for $ j \leq J$. Then there exists an absolute constant $ C>0$ such that
 $$ \int_{ \mathbb{R}^{2J}} \bigg| \log \bigg| \sum_{j \leq J} b_j e^{u_j + iv_j } \bigg|  \bigg|^{2k} e^{ -    \sum_j (u_j^2 + v_j^2 )/M} d \vecu d \vecv \ll M^{J+k} (Ck)^k + M^J (Ck)^{2k} .$$
 \end{lemma}
  We prove the propositions and the lemma in \S \ref{proof props}. Now we shall estimate the integral in \eqref{eqn int G}. By the Cauchy-Schwarz inequality, Proposition \ref{prop:G} and Lemma \ref{lemma}, we find that 
   \begin{align*}
 \mathbb{E} \big[  \log | F_J ( \sigma_T   :  X ) |  \big] = &  \int_{ [-M_1, M_1]^{2J} } \log \bigg|  \sum_{j \leq J} b_j e^{ u_j + i v_j } \bigg|  G_{\theta, T} (\vecu, \vecv)  d \vecu  d\vecv +O \bigg(  \frac{1}{ ( \log \log T)^{\eta}} \bigg)
 \end{align*}
 for any $ \eta>0$, where $M_1 =   \eta' \sqrt{ \log \log T \log \log \log T} $ with $ \eta'>0$ depending on $\eta$. By the Cauchy-Schwarz inequality and Lemma \ref{lemma}, we also have
    \begin{align*}
  \bigg( \int_{ [-M_1, M_1]^{2J} }& \log \bigg|  \sum_{j \leq J} b_j e^{ u_j + i v_j } \bigg|     d \vecu  d\vecv \bigg)^2  \\
  & \leq (2M_1)^{2J}  \int_{ [-M_1, M_1]^{2J} } \bigg(\log \bigg|  \sum_{j \leq J} b_j e^{ u_j + i v_j } \bigg|  \bigg)^2   d \vecu  d\vecv \\
  & \ll   M_1 ^{2J}  \int_{ \mathbb{R}^{2J} } \bigg(\log \bigg|  \sum_{j \leq J} b_j e^{ u_j + i v_j } \bigg|  \bigg)^2  e^{ - \sum_j ( u_j^2 + v_j^2 )/M_1^2 }   d \vecu  d\vecv\\
  &  \ll M_1^{4J+2}.
 \end{align*}
 Hence, by Proposition \ref{prop:G}
   \begin{align*}
 \mathbb{E}\big[  \log & | F_J ( \sigma_T   :  X ) |  \big] \\
 = &  \sum_{\veck, \vecl, \vecm, \vecn \in  (\mathbb{Z}_{\geq 0 })^J  }  \frac{ q_{ \veck, \vecl: \vecm, \vecn}    }{  ( \theta \log \log T)^{ J+ \mathcal{K}(\veck+\vecl+ \vecm+\vecn)  } }  \\
 & \quad \int_{ [-M_1, M_1]^{2J} } \log \bigg|  \sum_{j \leq J} b_j e^{ u_j + i v_j } \bigg|  \exp\bigg( -  \sum_j \frac{ u_j^2+v_j^2}{ \theta \xi_j \log\log T} \bigg) \vecu^{\vecm} \vecv^{\vecn}   d \vecu  d \vecv \\ 
 & +O \bigg(  \frac{( \log \log \log T)^{J+1/2} }{ ( \log \log T)^{5/2}} \bigg).
 \end{align*}
 Once again, by the Cauchy-Schwarz inequality, Proposition \ref{prop:G} and Lemma \ref{lemma}, we find that
   \begin{equation}\begin{split}\label{eqn:EFJI}
    \mathbb{E}\big[  \log&  | F_J ( \sigma_T   :  X ) |  \big] \\
    = &  \sum_{ \veck, \vecl, \vecm, \vecn \in  (\mathbb{Z}_{\geq 0 })^J }  \frac{ q_{ \veck, \vecl: \vecm, \vecn}    }{  ( \theta \log \log T)^{ J+ \mathcal{K}(\veck+\vecl+ \vecm+\vecn)  } }  \\
 & \quad \int_{ \mathbb{R}^{2J} } \log \bigg|  \sum_{j \leq J} b_j e^{ u_j + i v_j } \bigg|  \exp\bigg( -  \sum_j \frac{ u_j^2+v_j^2}{ \theta \xi_j \log\log T} \bigg)  \vecu^{\vecm} \vecv^{\vecn}   d \vecu  d\vecv \\ 
 & +O \bigg(  \frac{( \log \log \log T)^{J+1/2} }{ ( \log \log T)^{5/2}} \bigg) \\
  = &    \sum_{ \veck, \vecl, \vecm, \vecn \in  (\mathbb{Z}_{\geq 0 })^J  }  \frac{ q_{ \veck, \vecl: \vecm, \vecn}    }{      (  \theta\log \log T)^{ \mathcal{K}(\veck+\vecl)+ \mathcal{K}(\vecm+\vecn)/2 }} I_{  \vecm, \vecn}(\theta, T)   +O \bigg(  \frac{( \log \log \log T)^{J+1/2} }{ ( \log \log T)^{5/2}} \bigg),
 \end{split}\end{equation}
 where
 $$ I_{  \vecm, \vecn}(\theta, T)  : =\int_{ \mathbb{R}^{2J} } \log \bigg|  \sum_{j \leq J} b_j e^{ (u_j + i v_j)\sqrt{\theta \log \log T}  } \bigg|  e^{ -  \sum_j (u_j^2+v_j^2)/   \xi_j  }      \vecu^{\vecm} \vecv^{\vecn}      d \vecu  d\vecv . 
$$
  The logarithm is dominated by the biggest term in the $j$-sum, so that we divide $\mathbb{R}^{2J}$  into $J$ pieces
 \begin{align*}
I_{  \vecm, \vecn}(\theta, T)   = & \sum_{ \ell=1}^J  \int_{\mathbb{R}^{J}}    \int_{\mathcal{R}_{\ell }  } \log \bigg|  \sum_{j \leq J} b_j e^{ (u_j + i v_j)\sqrt{\theta  \log \log T}  } \bigg|  e^{ -  \sum_j (u_j^2+v_j^2)/ \xi_j }      \vecu^{\vecm} \vecv^{\vecn}    d\vecu d\vecv   ,
 \end{align*}
 where $\mathcal{R}_\ell$ is defined in \eqref{def Rell}.  By symmetry, it is enough to consider $\mathcal{R}_1 $. Then
 \begin{align*}
 \int_{\mathbb{R}^{J}}   \int_{\mathcal{R}_{1} }&   \log \bigg|  \sum_{j \leq J} b_j e^{ (u_j + i v_j)\sqrt{ \theta \log \log T}  } \bigg|  e^{ -  \sum_j (u_j^2+v_j^2)/ \xi_j }      \vecu^{\vecm} \vecv^{\vecn}    d\vecu d\vecv \\
 =&  \int_{\mathbb{R}^{J}}   \int_{\mathcal{R}_1 } \log \bigg|    b_1 e^{(u_1  +iv_1 )\sqrt{  \theta  \log \log T }} \bigg|     e^{ -  \sum_j (u_j^2+v_j^2)/\xi_j }      \vecu^{\vecm} \vecv^{\vecn}   d\vecu d\vecv +  \mathcal{E}_{\vecm, \vecn, 1 } (\theta, T)  \\
  =&    d_{\vecn}  \int_{\mathcal{R}_1 } (   \sqrt{   \theta  \log \log T}  u_1    +   \log |b_1| )     e^{ -  \sum_j  u_j^2 /\xi_j }      \vecu^{\vecm}     d\vecu   +  \mathcal{E}_{\vecm, \vecn, 1 } (\theta, T), 
  \end{align*}
 where
 \begin{align*}
 \mathcal{E}_{\vecm, \vecn, 1 } (\theta, T)  := &   \int_{\mathbb{R}^{J}}   \int_{\mathcal{R}_1}  \log \bigg|  1 +  \sum_{j=2}^J \frac{b_j}{b_1}  e^{(u_j -u_1 +i (v_j -v_1)) \sqrt{ \theta \log \log T }  }  \bigg| e^{ -  \sum_j (u_j^2+v_j^2)/ \xi_j }      \vecu^{\vecm} \vecv^{\vecn}    d\vecu d\vecv
 \end{align*}
 and
 \begin{equation}\label{eqn:dntheta}
  d_{\vecn} :=    \int_{\mathbb{R}^{J}}       e^{ -  \sum_j  v_j^2/ \xi_j  }  \vecv^{\vecn}    d\vecv  =   \prod_{j \leq J}   \bigg(   \xi_j^{(n_j+1)/2 }   \int_{\mathbb{R}}  v^{n_j} e^{- v^2 } dv \bigg) .
  \end{equation}
  Note that $ \int_{\mathbb{R}}  v^{n_j} e^{- v^2 } dv  = 0 $ if  $ n_j $ is odd, and $ = \Gamma ( ( n_j +1)/2) $ otherwise. Therefore,
 \begin{equation}\label{eqn:main int asymp}\begin{split}
I_{\vecm, \vecn}( \theta, T)  = &  \sum_{ \ell=1}^J  d_{\vecn}  \int_{\mathcal{R}_\ell } (   \sqrt{  \theta   \log \log T}  u_\ell     +   \log |b_\ell | )     e^{ -  \sum_j  u_j^2 /  \xi_j  }      \vecu^{\vecm}     d\vecu   + \mathcal{E}_{\vecm, \vecn } (\theta, T)     \\
= &   \sqrt{   \theta  \log \log T} d_{\vecn} \sum_{ \ell=1}^J   \int_{\mathcal{R}_\ell }     u_\ell         e^{ -  \sum_j  u_j^2 /  \xi_j  }       \vecu^{\vecm}     d\vecu  \\
& +  d_{\vecn}  \sum_{ \ell=1}^J   \int_{\mathcal{R}_\ell }     \log |b_\ell |       e^{ -  \sum_j  u_j^2 / \xi_j  }       \vecu^{\vecm}    d\vecu   + \mathcal{E}_{\vecm, \vecn } (\theta, T) ,
 \end{split}\end{equation}
where
 $$  \mathcal{E}_{\vecm, \vecn } (\theta, T)    := \sum_{\ell=1}^J  \mathcal{E}_{\vecm, \vecn, \ell } (\theta, T)     .$$
 By estimating $ \mathcal{E}_{\vecm, \vecn } (\theta, T) $ in \S \ref{sec proof prop3}, we prove the following proposition.
 \begin{prop} \label{prop3}
 Let $ J \geq 2 $. Then  we have
 $$    \mathcal{E}_{\vecm, \vecn } (\theta, T) = O\bigg(  \frac{ 1}{  ( \log \log T)^{1/4}}    \bigg)  .$$ 
 \end{prop}
 Therefore, by \eqref{eqn:EFJI}, \eqref{eqn:main int asymp} and Propositions \ref{prop3} and \ref{prop:G} we find that
 \begin{equation}\label{eqn:2}\begin{split}
     \mathbb{E} \big[ & \log  | F_J ( \sigma_T   :  X ) |  \big] \\
  = &     q_{0,0 : 0,0}       d_{0}     \bigg( \sqrt{   \theta \log \log T}  \sum_{ \ell=1}^J   \int_{\mathcal{R}_\ell }     u_\ell         e^{ -  \sum_j  u_j^2 /  \xi_j  }           d\vecu    +   \sum_{ \ell=1}^J   \int_{\mathcal{R}_\ell }     \log |b_\ell |       e^{ -  \sum_j  u_j^2 /  \xi_j  }            d\vecu \bigg)   \\
  & +   \sum_{    \mathcal{K}(  \vecm )=1 }    q_{0,0 : \vecm, 0}       d_{0}         \sum_{ \ell=1}^J   \int_{\mathcal{R}_\ell }     u_\ell         e^{ -  \sum_j  u_j^2 /  \xi_j  }      \vecu^{\vecm}       d\vecu         +  O\bigg(  \frac{ 1}{  ( \log \log T)^{1/4}}    \bigg)  \\
  = &      \pi^{-J/2} \prod_j \xi_j^{-1/2}   \sqrt{  \theta  \log \log T}  \sum_{ \ell=1}^J   \int_{\mathcal{R}_\ell }     u_\ell         e^{ -  \sum_j  u_j^2 /  \xi_j  }           d\vecu  \\
 &   +    \pi^{-J/2} \prod_j \xi_j^{-1/2}    \sum_{ \ell=1}^J     \log |b_\ell |    \int_{\mathcal{R}_\ell }      e^{ -  \sum_j  u_j^2 /  \xi_j }            d\vecu     +  O\bigg(  \frac{ 1}{  ( \log \log T)^{1/4}}    \bigg)  .
 \end{split}\end{equation}
 This proves Theorem \ref{thrm expt}. 
 
 
 Next we prove Theorem \ref{thrm zero density} assuming Conjecture \ref{conj}. By Littlewood's lemma and  \eqref{eqn:EFJI} we see that
\begin{equation}\label{eqn:Littlewood}\begin{split}
\int_{ \sigma_T(\theta_1)}^{\sigma_T(\theta_2)} & N_{F_J} ( w : T)dw 
\\ =& \frac{1}{ 2 \pi} \int_T^{2T} \log |F_J ( \sigma_T(\theta_1) +it)|dt \\
&  - \frac{1}{ 2 \pi} \int_T^{2T} \log |F_J ( \sigma_T(\theta_2) +it)|dt  + O\bigg(  \frac{T}{ ( \log T)^{\theta_2}}\bigg) \\
 =  &  \frac{T}{2\pi} \bigg(  \mathbb{E} \big[  \log | F_J ( \sigma_T  (\theta_1)  :  X ) |  \big]  -  \mathbb{E} \big[  \log | F_J ( \sigma_T (\theta_2)  :  X ) |  \big]  \bigg) + O \bigg(  \frac{T}{ ( \log T)^\eta }\bigg),
\end{split}\end{equation}
where $ \sigma_T( \theta) = 1/2+ (\log T)^{ - \theta }$ and $ 0 < \theta_2 < \theta_1 $.  We need the following lemma.
\begin{lemma}\label{lemma zero density}
Let $\alpha $ be a real number, $ \theta_1>\theta_2>0$ and $H_T = \theta_1 - \theta_2 $. Suppose that $H_T   \to 0 $ as $T \to \infty$. Then  for each $i =1,2$ we have
\begin{equation*} \begin{split}
  \theta_1^\alpha  I_{  \vecm, \vecn}& ( \theta_1 , T)  - \theta_2^\alpha  I_{  \vecm, \vecn}( \theta_2 , T)  \\
= &  H_T   \sqrt{    \log \log T} d_{\vecn}       \bigg(   \alpha+\frac12  \bigg)  \theta_i^{    \alpha  -1/2   }    \sum_{ \ell=1}^J   \int_{\mathcal{R}_\ell }     u_\ell         e^{ -  \sum_j  u_j^2 /    \xi_j  }     \vecu^{\vecm}     d\vecu  \\
& + H_T    d_{\vecn}  \alpha  \theta_i^{  \alpha    -1 }     \sum_{ \ell=1}^J   \int_{\mathcal{R}_\ell }     \log |b_\ell |       e^{ -  \sum_j  u_j^2 /   \xi_j  }       \vecu^{\vecm}     d\vecu \\
&   +  O\bigg(  \frac{H_T  }{ ( \log \log T)^{1/4}} +  H_T  ^2 \sqrt{ \log \log T}  \bigg)  .
 \end{split}\end{equation*}
\end{lemma}
We prove it in \S \ref{sec proof zero density lemma}. Suppose that $H_T \log \log T = o(1)$, then 
 \begin{align*}
 \int_{\sigma_T(\theta)}^{ \sigma_T( \theta-H_T  )} N_{F_J} ( w : T) dw&  \leq \big( \sigma_T ( \theta-H_T  ) - \sigma_T(\theta) \big) N_{F_J} ( \sigma_T( \theta):T) \\
 & = \frac{ H_T   \log \log T}{ ( \log T)^\theta}  \big( 1+ O(H_T   \log \log T)\big) N_{F_J} ( \sigma_T( \theta):T) 
 \end{align*}
 and
 \begin{align*}
  \int_{\sigma_T(\theta+H_T   )}^{ \sigma_T( \theta)} N_{F_J} ( w : T) dw&  \geq \big( \sigma_T ( \theta) - \sigma_T(\theta+H_T   ) \big) N_{F_J} ( \sigma_T( \theta):T) \\
  & = \frac{ H_T   \log \log T}{ ( \log T)^\theta} \big( 1+ O( H_T   \log \log T)\big) N_{F_J} ( \sigma_T( \theta):T).
 \end{align*}
 By \eqref{eqn:Littlewood}, \eqref{eqn:EFJI} and Lemma \ref{lemma zero density}, we find that
 \begin{align*}
  \frac{ H_T   \log \log T}{ ( \log T)^\theta} & ( 1+ O(H_T   \log \log T)) N_{F_J} ( \sigma_T( \theta):T) \\
=&\frac{  H_T    T }{2 \pi}   \sum_{  \veck, \vecl,\vecm, \vecn \in ( \mathbb{Z}_{\geq 0 } )^J }  \frac{ q_{ \veck, \vecl: \vecm, \vecn}    }{     ( \theta \log \log T)^{ \mathcal{K}(\veck+\vecl)+ \mathcal{K}(\vecm+\vecn)/2 }} \\
& \bigg(  \sqrt{    \log \log T} d_{\vecn}    \bigg( -\mathcal{K}( \veck+\vecl)  + \frac{ 1 - \mathcal{K}(\vecm+ \vecn) }{2}  \bigg)  \theta^{ -1/2   }    \sum_{ \ell=1}^J   \int_{\mathcal{R}_\ell }     u_\ell         e^{ -  \sum_j  u_j^2 /    \xi_j  }       \vecu^{\vecm}     d\vecu  \\
& +   d_{\vecn}   \bigg( -\mathcal{K}( \veck+\vecl)- \frac{ \mathcal{K}(\vecm+ \vecn) }{2}   \bigg)  \theta^{     -1 }     \sum_{ \ell=1}^J   \int_{\mathcal{R}_\ell }     \log |b_\ell |       e^{ -  \sum_j  u_j^2 /   \xi_j  }       \vecu^{\vecm}     d\vecu \bigg)  \\
&  +O \bigg(  \frac{T ( \log \log \log T)^{J+1/2} }{ ( \log \log T)^{5/2}}+   \frac{ H_T T}{ ( \log \log T)^{1/4}} + H_T^2 T \sqrt{ \log \log T}   \bigg) .
\end{align*} 
Choose $H_T = ( \log \log T)^{-2}$ to optimize it, we see that 
 \begin{align*}
     N_{F_J} ( \sigma_T( \theta):T) 
=&\frac{       T( \log T)^\theta }{2 \pi}     \sum_{  \veck, \vecl,\vecm, \vecn \in ( \mathbb{Z}_{\geq 0 } )^J  }  \frac{ q_{ \veck, \vecl: \vecm, \vecn}  d_{\vecn}  }{     (\theta \log \log T)^{1+\mathcal{K}(\veck+\vecl)+ \mathcal{K}(\vecm+\vecn)/2 }} \\
& \bigg(  \sqrt{ \theta    \log \log T}     \bigg( -\mathcal{K}( \veck+\vecl)  + \frac{ 1 - \mathcal{K}(\vecm+ \vecn) }{2}  \bigg)      \sum_{ \ell=1}^J   \int_{\mathcal{R}_\ell }     u_\ell         e^{ -  \sum_j  u_j^2 /    \xi_j  }       \vecu^{\vecm}     d\vecu  \\
& +     \bigg( -\mathcal{K}( \veck+\vecl)- \frac{ \mathcal{K}(\vecm+ \vecn) }{2}   \bigg)     \sum_{ \ell=1}^J   \int_{\mathcal{R}_\ell }     \log |b_\ell |       e^{ -  \sum_j  u_j^2 /   \xi_j  }       \vecu^{\vecm}     d\vecu \bigg)  \\
&  +O \bigg( \frac{ T( \log T)^\theta}{   ( \log \log T)^{5/4} }  \bigg) .
\end{align*} 
We see that the summands are smaller than the $O$-term unless $\veck = \vecl = 0 $ and $\mathcal{K}( \vecm+\vecn)= 0, 1 $. Moreover, $ q_{0,0: \vecm, \vecn} = 0 $ if $\mathcal{K}( \vecm+\vecn)=  1 $. Hence, 
 \begin{align*}
   N_{F_J} ( \sigma_T( \theta):T)  
=&  \frac{T ( \log T)^\theta}{  \sqrt{\theta  \log \log T} }       \frac{    q_{0,0:0,0} d_0      }{4 \pi}       \sum_{ \ell=1}^J   \int_{\mathcal{R}_\ell }     u_\ell         e^{ -  \sum_j  u_j^2 /    \xi_j  }          d\vecu    +O \bigg( \frac{ T( \log T)^\theta}{   ( \log \log T)^{5/4} }  \bigg) .
\end{align*} 
Since
 $$q_{0,0:0,0} d_0  = \pi^{-J/2} \prod_j  \xi_j ^{-1/2}, $$
we prove the theorem.
 
 \section{Proof of propositions and lemmas}\label{proof props}
 
 \subsection{Proof of Proposition \ref{prop:hat Psi}}
 Let $ z_j = \pi ( x_j + i y_j ) $ for $ j = 1, \ldots, J$, then
\begin{align*}
\hat{\Psi}_{\theta, T} ( \vecx , \vecy) & = \mathbb{E} \bigg[ \exp \bigg( 2 \pi i  \sum_{j \leq J } (  x_j \Re \log L_j ( \sigma_T : X ) + y_j \Im \log L_j ( \sigma_T : X )    )  \bigg)\bigg] \\
 & = \mathbb{E} \bigg[ \exp \bigg( 2 \pi i  \sum_{j \leq J }  \Re \big[(x_j - i y_j )   \log L_j ( \sigma_T : X ) \big]    \bigg)\bigg] \\
  & = \mathbb{E} \bigg[ \exp \bigg(   i  \sum_{j \leq J } \bar{z}_j   \log L_j ( \sigma_T : X ) + z_j \log L_j ( \sigma_T : \bar{X} )    \bigg)\bigg] .
\end{align*}
 Write 
 $$ \log L_j ( \sigma : X ) := \sum_p g_j ( p, \sigma:X), \qquad g_j ( p, \sigma:X) : =  \sum_{n=1}^\infty   \frac{ a_j ( p^n ) X(p)^n }{ p^{n \sigma}}, $$
 then
 \begin{align*}
\hat{\Psi}_{\theta, T} ( \vecx , \vecy)    & =\prod_p  \mathbb{E} \bigg[ \exp \bigg(   i  \sum_{j \leq J } \bar{z}_j   g_j (  p, \sigma_T : X ) + z_j g_j ( p, \sigma_T : \bar{X} )    \bigg)\bigg] \\
  & =\prod_p  \mathbb{E} \bigg[ \prod_{j \leq J }  \exp \bigg(   i  \big(  \bar{z}_j   g_j (  p, \sigma_T : X ) + z_j g_j ( p, \sigma_T : \bar{X} )   \big)  \bigg)\bigg] 
  \end{align*}
and we see that
\begin{equation}\label{eqn:exp bound 1}
 \bigg|  \mathbb{E} \bigg[ \prod_{j \leq J }  \exp \bigg(   i  \big(  \bar{z}_j   g_j (  p, \sigma_T : X ) + z_j g_j ( p, \sigma_T : \bar{X} )   \big)  \bigg)\bigg] \bigg| \leq 1 .
 \end{equation}
By Lemma 2.5 in \cite{Le2} and the argument to justify the equation (3.28) in \cite[p. 1828 --1829]{Le2}, there is a constant $C_1 >0$ such that
$$ \bigg|  \mathbb{E} \bigg[ \prod_{j \leq J }  \exp \bigg(   i  \big(  \bar{z}_j   g_j (  p, \sigma_T  : X ) + z_j g_j ( p, \sigma_T : \bar{X} )   \big)  \bigg)\bigg]  \bigg| \leq   C_1  \frac{p^{\sigma_T /2} }{     \big(  \sum_j x_j^2 + y_j^2 \big)^{1/4} } $$
for $ p^{-\sigma_T }  \sqrt{  \sum_j x_j^2 + y_j^2 } \geq 1 $. Hence if $ p^{\sigma_T}  \leq C_2  \sqrt{  \sum_j x_j^2 + y_j^2 } $ with $ C_2  = \min\{ 1, C_1 ^{-2} e^{-1} \}$, then
\begin{equation}\label{eqn:exp bound 2}
\bigg|  \mathbb{E} \bigg[ \prod_{j \leq J }  \exp \bigg(   i  \big(  \bar{z}_j   g_j (  p, \sigma_T  : X ) + z_j g_j ( p, \sigma_T : \bar{X} )   \big)  \bigg)\bigg]  \bigg| \leq  e^{-1/2}. 
\end{equation}
Thus, by \eqref{eqn:exp bound 1}, \eqref{eqn:exp bound 2} and the prime number theorem, we have
$$   | \hat{\Psi}_{\theta,T} ( \vecx, \vecy) |    \leq   \prod_{   p^{\sigma_T}  \leq C_2  \sqrt{  \sum_j x_j^2 + y_j^2 }} e^{-1/2}    \leq \exp \bigg(  - C_3   \frac{ ( \sum_j x_j^2 + y_j^2)^{1/2\sigma_T} }{ \log (\sum_j x_j^2 + y_j^2)    }  \bigg)  $$
for $ \sum_j x_j^2 + y_j^2 \geq C_4  $ and for some $ C_3, C_4  >0$. This proves the first inequality in Proposition \ref{prop:hat Psi}.

Let
\begin{equation} \label{def:Aklpsigma}
 A_{\veck, \vecl} (p,\sigma) :=   \mathbb{E} \bigg[       \prod_{j\leq J}    g_j (  p, \sigma : X )^{k_j}  g_j ( p, \sigma: \bar{X} )^{\ell_j}   \bigg]  
 \end{equation}
for $ \veck = ( k_1 , \ldots, k_J)$ and $ \vecl = ( \ell_1, \ldots, \ell_J)$, then each factor of $ \hat{\Psi}_{\theta, T} ( \vecx , \vecy)  $ is 
  \begin{align*}
  &     \sum_{\veck, \vecl \geq 0} \frac{i^{\mathcal{K}(\veck + \vecl)  } }{ \veck ! \vecl ! }    \mathbb{E} \bigg[       \prod_{j\leq J}    g_j (  p, \sigma_T : X )^{k_j}  g_j ( p, \sigma_T : \bar{X} )^{\ell_j}    \big)  \bigg]   \bar{\vecz}^{\veck} \vecz^{\vecl  }   \\
& =     \sum_{\veck, \vecl \geq 0} \frac{i^{\mathcal{K}(\veck + \vecl)  } }{ \veck ! \vecl ! }   A_{\veck, \vecl} (p,\sigma_T )   \bar{\vecz}^{\veck} \vecz^{\vecl  }  ,
\end{align*}
where the sums are over all $ \veck = ( k_1 , \ldots, k_J), \vecl = ( \ell_1, \ldots, \ell_J) \in (\mathbb{Z}_{\geq 0 })^J $ and  $\mathcal{K}(\veck)= k_1 + \cdots + k_J $, $ \veck ! = k_1 ! \cdots k_J ! $ and $  \vecz^{\veck  } = \prod_{j\leq J}  z_j^{k_j}  $. 
Since $ A_{0,0} ( p,\sigma) = 1 $ and  $ A_{0, \veck } ( p, \sigma) = A_{  \veck, 0  } ( p, \sigma) = 0 $ for $ \veck \neq 0 $, the above sum equals
$$  1 +  \sumstar_{\veck, \vecl} \frac{i^{\mathcal{K}(\veck + \vecl)  } }{ \veck ! \vecl ! }   A_{\veck, \vecl} (p,\sigma_T )    \bar{\vecz}^{\veck} \vecz^{\vecl  } =: 1+   \mathcal{A}_{\theta, T, \vecx, \vecy} (p), $$  
where  the $*$-sum is over all nonzero $ \veck , \vecl \in (\mathbb{Z}_{\geq 0 })^J $.  
By estimating \eqref{def:Aklpsigma} one can show that
$$  | A_{\veck, \vecl} ( p , \sigma_T) | \leq  C_5 p^{ - \sigma_T  \mathcal{K}( \veck + \vecl)} $$
for some $C_5>0$ and all nonzero $ \veck , \vecl \in (\mathbb{Z}_{\geq 0 })^J   $. Thus,
\begin{equation}\label{eqn:Att}
  |  \mathcal{A}_{\theta, T, \vecx, \vecy} (p)|    \leq C_5   \sumstar_{\veck, \vecl} \frac{1 }{ \veck ! \vecl ! }   \prod_{j\leq J} \bigg(    \frac{   | z_j |}{p^{ \sigma_T}}    \bigg)^{k_j+ \ell_j }  = C_5 \bigg( \exp\bigg(  \frac{ \sum_{j=1}^J |z_j|  }{ p^{ \sigma_T} } \bigg)   -1   \bigg)^2   . 
  \end{equation}
Let $Y = e^{ ( \log T)^{ \theta /2 }} $, then there exists a constant  $ C_6 > 0 $ such that 
\begin{align*}
 |  \mathcal{A}_{\theta, T, \vecx, \vecy} (p)|  \leq C_6   \frac{\sum_{j=1}^J |z_j|^2 }{p^{2 \sigma_T}} \leq   C_6
\end{align*}
for $ \sum_{ j \leq J } |z_j|^2 \leq  Y $ and $ p \geq Y$. Thus, by \eqref{eqn:exp bound 1}
\begin{align*}
| \hat{\Psi}_{\theta,T} ( \vecx, \vecy) |  &   \leq   \bigg|  \prod_{   p \geq Y } ( 1+   \mathcal{A}_{\theta, T, \vecx, \vecy} (p)  ) \bigg|  \\
& =\bigg| \prod_{p \geq Y} \exp \bigg(    \mathcal{A}_{\theta, T, \vecx, \vecy} (p)  + O \bigg(  \frac{\big( \sum_{j=1}^J |z_j|^2  \big)^2 }{p^{4 \sigma_T}}   \bigg)  \bigg)\bigg| \\
& =\bigg| \exp \bigg(   \sum_{ p \geq Y}   \mathcal{A}_{\theta, T, \vecx, \vecy} (p)  + O \bigg(    \sum_{j=1}^J |z_j|^2       \bigg) \bigg)\bigg| . 
\end{align*}
The $p$-sum is
\begin{align*}
 \sum_{ p \geq Y}   \mathcal{A}_{\theta, T, \vecx, \vecy} (p)   & =   \sum_{p \geq Y}  \sumstar_{\veck, \vecl} \frac{i^{\mathcal{K}(\veck + \vecl)  } }{ \veck ! \vecl ! }   A_{\veck, \vecl} (p,\sigma_T )   \bar{\vecz}^{\veck} \vecz^{\vecl  } \\
  & = -  \sum_{p \geq Y}   \sum_{j_1, j_2 \leq J }   \mathbb{E} \bigg[          g_{j_1}  (  p, \sigma_T : X )   g_{j_2} ( p, \sigma_T : \bar{X} )    \big)  \bigg] \bar{z}_{j_1} z_{j_2} + O  \bigg(  \sum_{j=1}^J |z_j|^2   \bigg)\\
 & = -    \sum_{j_1, j_2 \leq J } \sum_{p \geq Y}    \frac{ a_{j_1} ( p) a_{j_2} (p) }{ p^{2 \sigma_T} } \bar{z}_{j_1} z_{j_2}     + O  \bigg(  \sum_{j=1}^J |z_j|^2   \bigg)\\
 &= - \sum_{j \leq J } |z_j|^2 \bigg( \sum_{p \geq Y}    \frac{ a_{j } ( p)^2  }{ p^{2 \sigma_T} }      + O(1) \bigg) \\
 &\leq  -   \bigg(  \frac{ \pi^2 \theta}{2}  \log \log T  + O(1) \bigg)  \sum_{j \leq J }  (x_j^2 + y_j^2 )  .
\end{align*}
Therefore,
\begin{align*}
| \hat{\Psi}_{\theta,T} ( \vecx, \vecy) |  & \leq  \exp \bigg(    -   \bigg(  \frac{ \pi^2 \theta}{2}  \log \log T  + O(1) \bigg)  \sum_{j \leq J }  (x_j^2 + y_j^2 ) \bigg)  
\end{align*}
holds for  $ \sum_{ j \leq J } (x_j^2 + y_j^2 )  \leq   e^{ ( \log T)^{ \theta /2 }}  $, which proves the second inequality in Proposition  \ref{prop:hat Psi}.

Next, we find an asymptotic of $\hat{\Psi}_{\theta,T} ( \vecx, \vecy)$ for $ \sum_{ j \leq J } (x_j^2 + y_j^2 ) \leq C_7 $. By \eqref{eqn:Att} and choosing $C_7>0 $ sufficiently small, we have that
$ |    \mathcal{A}_{\theta, T, \vecx, \vecy} (p) | \leq 1/2 $ for every prime $p$. Thus,
\begin{align*}
 \hat{\Psi}_{\theta,T} ( \vecx, \vecy) &  = \prod_{   p  } ( 1+   \mathcal{A}_{\theta, T, \vecx, \vecy} (p)  )   \\
& =\prod_p   \exp \bigg(     \mathcal{A}_{\theta, T, \vecx, \vecy} (p)  - \frac12  ( \mathcal{A}_{\theta, T, \vecx, \vecy} (p) )^2    + O \bigg(  \frac{ \sum_{j=1}^J (x_j^2 + y_j^2 )^3 }{p^{6 \sigma_T}}   \bigg)  \bigg)  \\
& =  \exp \bigg(     \sum_p  \mathcal{A}_{\theta, T, \vecx, \vecy} (p)  - \frac12  \sum_p  ( \mathcal{A}_{\theta, T, \vecx, \vecy} (p) )^2    + O \bigg(   \sum_{j=1}^J (x_j^2 + y_j^2 )^3   \bigg)  \bigg) .
\end{align*}
The sum 
$$    \sum_p  \mathcal{A}_{\theta, T, \vecx, \vecy} (p)  - \frac12  \sum_p  ( \mathcal{A}_{\theta, T, \vecx, \vecy} (p) )^2   $$
has a power series representation in $z_1, \bar{z}_1 , \ldots, z_J, \bar{z}_J$, so let it be
$$  \sumstar_{\veck, \vecl}     B_{\veck, \vecl} (\sigma_T )    \bar{\vecz}^{\veck} \vecz^{\vecl  } .$$
For  $\mathcal{K}( \veck+\vecl) \geq 3 $, we have
$$ B_{\veck, \vecl} (\sigma_T )  = B_{\veck, \vecl} (1/2 ) + O\bigg(   \frac{1}{ ( \log T)^\theta} \bigg).$$
For $\mathcal{K}( \veck)=\mathcal{K}(\vecl)=1$, we have
\begin{align*}
 \sum_{\mathcal{K}( \veck)=\mathcal{K}(\vecl)=1}  B_{\veck, \vecl} (\sigma_T )    =  &    \sum_p   \sum_{\mathcal{K}( \veck)=\mathcal{K}(\vecl)=1}  \frac{i^{\mathcal{K}(\veck + \vecl)  } }{ \veck ! \vecl ! }   A_{\veck, \vecl} (p,\sigma_T )   \bar{\vecz}^{\veck} \vecz^{\vecl  } \\
   = &   -  \sum_p  \sum_{j_1, j_2 \leq J }   \mathbb{E} \bigg[          g_{j_1}  (  p, \sigma_T : X )   g_{j_2} ( p, \sigma_T : \bar{X} )       \bigg] \bar{z}_{j_1} z_{j_2}  .
   \end{align*}
By \eqref{eqn:soc},  we find that
$$  \mathbb{E} \bigg[          g_{j_1}  (  p, \sigma_T : X )   g_{j_2} ( p, \sigma_T : \bar{X} )    \big)  \bigg] 
  = C_{j_1, j_2} +  O\bigg(   \frac{1}{ ( \log T)^\theta} \bigg) $$
   for $ j_1 \neq j_2 $ and  
   \begin{align*}
   \mathbb{E} \bigg[          g_j  (  p, \sigma_T : X )   g_j ( p, \sigma_T : \bar{X} )    \big)  \bigg] &  = \sum_p    \frac{ a_{j } ( p)^2  }{ p^{2 \sigma_T} }  +  C'_j +  O\bigg(   \frac{1}{ ( \log T)^\theta} \bigg) \\
   & = \xi_j \theta \log \log T + C_{j,j} +O\bigg(   \frac{1}{ ( \log T)^\theta} \bigg) 
   \end{align*}
   for some constants $ C_{j_1, j_2}$, $C_{j, j} $, $C'_j$ independent to $\theta$. Thus, 
$$  \sum_{\mathcal{K}( \veck)=\mathcal{K}(\vecl)=1}  B_{\veck, \vecl} (\sigma_T )   = - \theta \log \log T  \sum_{j \leq J} \xi_j |z_j|^2 + \sum_{j_1, j_2 \leq J} C_{j_1, j_2} \bar{z}_{j_1} z_{j_2}    +  O\bigg(   \frac{1}{ ( \log T)^\theta} \bigg)  .$$
   Therefore, we have
   \begin{align*}
 \hat{\Psi}_{\theta,T} ( \vecx, \vecy)  = &   e^{   - \theta \log \log T  \sum_{j \leq J} \xi_j |z_j|^2} \exp\bigg(   \sum_{j_1, j_2 \leq J} C_{j_1, j_2} \bar{z}_{j_1} z_{j_2} +  \sum_{\mathcal{K}( \veck+ \vecl)=3,4,5 }     B_{\veck, \vecl} (1/2)   \bar{\vecz}^{\veck} \vecz^{\vecl  }   \bigg) \\
&  \times \exp\bigg(   O\bigg(      \sum_{j=1}^J (x_j^2 + y_j^2 )^3   + \frac{1}{ ( \log T)^\theta} \bigg)  \bigg) \\
= &   e^{   - \pi^2 \theta \log \log T  \sum_{j \leq J} \xi_j (x_j^2 + y_j^2 )}    \bigg( P( \vecx, \vecy)  +  O\bigg(       \sum_{j=1}^J (x_j^2 + y_j^2  )^3  + \frac{1}{ ( \log T)^\theta}   \bigg)  \bigg)   
\end{align*}
for  $ \sum_{ j \leq J } (x_j^2 + y_j^2 ) \leq C_7 $, where $ P(\vecx, \vecy)$ is a polynomial of degree $\leq 5 $ and may be written as 
$$ 1 +    \sum_{\mathcal{K}( \veck+ \vecl)= 2, 3,4,5 }     \tilde{B}_{\veck, \vecl}   \vecx^{\veck} \vecy^{\vecl } . $$
    This completes the proof of Proposition \ref{prop:hat Psi}.

  \subsection{Proof of Proposition \ref{prop:G}}
     By Proposition \ref{prop:hat Psi}, we have
\begin{align*}
 G_{\theta, T}  ( \vecu, \vecv) 
 = & \int_{ \mathbb{R}^{2J}} \hat{\Psi}_{ \theta, T} ( \vecx, \vecy) e^{- 2 \pi i ( \vecx \cdot \vecu + \vecy \cdot \vecv )} d \vecx d\vecy \\
  = &   \int_{\sum_j ( x_j^2 + y_j ^2 ) \leq C_7 } e^{ - \pi^2 \theta \log \log T \sum_{j \leq J} \xi_j  ( x_j ^2 + y_j^2 ) - 2 \pi i ( \vecx \cdot \vecu + \vecy \cdot \vecv )  }P ( \vecx  , \vecy)    d \vecx d \vecy  \\
 & +  O\bigg(   \frac{1}{ ( \log \log T)^{J+3}}  \bigg) \\
   = &    \int_{\mathbb{R}^{2J} } e^{ - \pi^2 \theta \log \log T \sum_{j \leq J} \xi_j  ( x_j ^2 + y_j^2 ) - 2 \pi i ( \vecx \cdot \vecu + \vecy \cdot \vecv )  }P ( \vecx  , \vecy)   d \vecx d \vecy \\
   &  +  O\bigg(   \frac{1}{ ( \log \log T)^{J+3}}  \bigg) ,
 \end{align*} 
 where $P ( \vecx  , \vecy) $ is the polynomial defined in Proposition \ref{prop:hat Psi}. By the change of variables
 $$ x_j   =  \frac{ \tilde{x}_j}{ \pi \sqrt{ \theta \xi_j \log \log T}} - \frac{ i u_j }{ \pi \theta \xi_j \log \log T} $$
 and
  $$ y_j   =  \frac{ \tilde{y}_j}{ \pi \sqrt{ \theta \xi_j \log \log T}} - \frac{ i v_j }{ \pi \theta \xi_j \log \log T},$$
one finds that
$$ P( \vecx, \vecy) =  \sum_{  \veck,\vecl,\vecm, \vecn \in (\mathbb{Z}_{\geq 0 })^J  } \frac{ p_{\veck, \vecl: \vecm,\vecn} }{  ( \theta \log \log T)^{ \mathcal{K}(\veck+\vecl)/2 + \mathcal{K} ( \vecm+\vecn)  } }   \tilde{\vecx}^{\veck} \tilde{\vecy}^{\vecl}  \vecu^{\vecm} \vecv^{\vecn}    ,$$
where  
 $$ p_{ \veck, \vecl: \vecm, \vecn}  =\frac{\tilde{B}_{ \veck+\vecm, \vecl+\vecn} }{ \pi^{ \mathcal{K}( \veck+\vecl+\vecm+\vecn)} i^{\mathcal{K}(\vecm+\vecn)}} \frac{ ( \veck+\vecm )!     (\vecl+\vecn)!}{ \veck! \vecm! \vecl! \vecn! }   \prod_j \xi^{ - ( k_j + \ell_j ) /2 - m_j  - n_j } .    $$
 Then we see that  
 \begin{align*}
 G_{\theta, T} ( \vecu, \vecv) =& \exp\bigg( -  \sum_j \frac{ u_j^2+v_j^2}{ \theta \xi_j \log\log T} \bigg) \frac{1}{ \prod_j ( \pi^2 \theta \xi_j \log \log T)}  \\
 & \int_{ \mathbb{R}^{2J}} e^{ - \sum_j ( \tilde{x}_j^2 + \tilde{y}_j^2 ) }   \sum_{ \veck,\vecl,\vecm, \vecn \in (\mathbb{Z}_{\geq 0 })^J } \frac{ p_{\veck, \vecl: \vecm,\vecn} }{  ( \theta \log \log T)^{ \mathcal{K}(\veck+\vecl)/2 + \mathcal{K} ( \vecm+\vecn)  } }  \tilde{\vecx}^{\veck} \tilde{\vecy}^{\vecl}  \vecu^{\vecm} \vecv^{\vecn} d \tilde{\vecx} d \tilde{\vecy} \\
 &  +  O\bigg(   \frac{1}{ ( \log \log T)^{J+3}}  \bigg) \\
 =& \exp\bigg( -  \sum_j \frac{ u_j^2+v_j^2}{ \theta \xi_j \log\log T} \bigg)   \sum_{  \veck,\vecl,\vecm, \vecn \in (\mathbb{Z}_{\geq 0 })^J }  \frac{ c_{\veck, \vecl} p_{\veck, \vecl: \vecm,\vecn}    }{  ( \theta \log \log T)^{ J+ \mathcal{K}(\veck+\vecl)/2 + \mathcal{K} ( \vecm+\vecn)  } } \vecu^{\vecm} \vecv^{\vecn}    \\
 &  +  O\bigg(   \frac{1}{ ( \log \log T)^{J+3}}  \bigg) ,
 \end{align*}
where
\begin{align*}
   c_{\veck, \vecl}:= &  \pi^{-2J} \bigg( \prod_j  \xi_j^{-1}  \bigg)  \int_{ \mathbb{R}^{2J}} e^{ - \sum_j ( \tilde{x}_j^2 + \tilde{y}_j^2 ) }    \tilde{\vecx}^{\veck} \tilde{\vecy}^{\vecl}  d \tilde{\vecx} d \tilde{\vecy}  \\
   = &  \pi^{-2J} \prod_j  \bigg( \xi_j^{-1} \int_{ \mathbb{R}} e^{ - x^2  } x^{k_j} dx  \int_{ \mathbb{R}} e^{ - y^2  } y^{\ell_j} dy   \bigg) .
\end{align*}
Thus, if there is odd $k_j$ or odd $\ell_j$, then $ c_{\veck, \vecl} = 0 $. Otherwise,
$$ c_{\veck, \vecl}= \pi^{-2J} \prod_j  \bigg( \xi_j^{-1} \Gamma \bigg( \frac{k_j +1}{2} \bigg)\Gamma \bigg( \frac{\ell_j +1}{2} \bigg)   \bigg)  .$$
 Hence, we have
  \begin{align*}
 G_{\theta, T} ( \vecu, \vecv) =&   \exp\bigg( -  \sum_j \frac{ u_j^2+v_j^2}{ \theta \xi_j \log\log T} \bigg)   \sum_{  \veck,\vecl,\vecm, \vecn \in (\mathbb{Z}_{\geq 0 })^J }  \frac{ c_{2\veck, 2\vecl} p_{2\veck, 2\vecl: \vecm,\vecn}    }{  ( \theta \log \log T)^{ J+ \mathcal{K}(\veck+\vecl+ \vecm+\vecn)  } }  \vecu^{\vecm} \vecv^{\vecn}    \\
 &  +  O\bigg(   \frac{1}{ ( \log \log T)^{J+3}}  \bigg) .
 \end{align*}
Letting $ q_{\veck, \vecl:\vecm,\vecn} = c_{2\veck, 2\vecl} p_{2\veck, 2\vecl: \vecm,\vecn}   $, we prove the first identity of the proposition. The second one can be deduced by modifying the proof of Theorem 6 in \cite{BJ}.

\subsection{Proof of Lemma \ref{lemma}}
Our proof is basically the same as the proof of Lemma 3.3 in \cite{GL}, but we need the dependency on $M$. We first see that 
\begin{align*}
\int_{ \mathbb{R}^{2J}} & \bigg| \log \bigg| \sum_{j \leq J} b_j e^{u_j + iv_j } \bigg|  \bigg|^{2k} e^{ -    \sum_j (u_j^2 + v_j^2 )/M} d \vecu d \vecv  \\
=& \int_{ \mathbb{R}^{J}} \int_{ [0,2 \pi]^J}  \sum_{\veck \in \mathbb{Z}^J}  \bigg| \log \bigg| \sum_{j \leq J} b_j e^{u_j + i(v_j + 2 \pi k_j ) } \bigg|  \bigg|^{2k} e^{ -    \sum_j (u_j^2 + (v_j + 2 \pi k_j )^2 )/M} d \vecv d \vecu  \\
=& \int_{ \mathbb{R}^{J}} \int_{ [0,2 \pi]^J}  \bigg| \log \bigg| \sum_{j \leq J} b_j e^{u_j + iv_j } \bigg|  \bigg|^{2k}   \sum_{\veck \in \mathbb{Z}^J}    e^{ -    \sum_j   (v_j + 2 \pi k_j )^2 /M} d \vecv  e^{ -    \sum_j u_j^2/M}    d \vecu  \\
\ll &  M^{J/2}   \int_{ \mathbb{R}^{J}} \int_{ [0,2 \pi]^J}  \bigg| \log \bigg| \sum_{j \leq J} b_j e^{u_j + iv_j } \bigg|  \bigg|^{2k}    d \vecv  e^{ -    \sum_j u_j^2/M}    d \vecu  .
\end{align*}
Next we need the inequality
 $$ \int_0^{2 \pi} \big( \log |a-be^{iv} | \big)^{2k} dv \ll ( C_1 \log |a| )^{2k} + ( C_1 \log |b| )^{2k} + (C_1 k)^{2k} $$
 for some constant $C_1 > 0$. (See Lemma 2.1 in \cite{GL} for a proof.) Hence, we see that
 \begin{align*}
  \int_{ \mathbb{R}^{J}}& \int_{ [0,2 \pi]^J}  \bigg| \log \bigg| \sum_{j \leq J} b_j e^{u_j + iv_j } \bigg|  \bigg|^{2k}    d \vecv  e^{ -    \sum_j u_j^2/M}    d \vecu \\
\ll &     \int_{ \mathbb{R}^{J}} \bigg( \sum_{j \leq J} \big( C_2 \log |b_j e^{u_j}| \big)^{2k}   + (C_2 k )^{2k}  \bigg)    e^{ -    \sum_j u_j^2/M}    d \vecu  \\
\ll &      \sum_{j \leq J}  \int_{ \mathbb{R}^{J}} \big(  ( C_3 u_j )^{2k} + C_3^{2k}  \big)      e^{ -    \sum_j u_j^2/M}    d \vecu  + M^{J/2} (C_2 k )^{2k}  \\
\ll &  M^{J/2} \sum_{j \leq J}  \int_{ \mathbb{R}^{J}} \big(  ( C_3 u_j )^{2k}M^k  + C_3^{2k}  \big)      e^{ -    \sum_j u_j^2 }    d \vecu  + M^{J/2} (C_2 k )^{2k} \\
\ll & M^{J/2} (C_4 kM)^k + M^{J/2} (C_2 k )^{2k}.
\end{align*}
Thus, 
$$ \int_{ \mathbb{R}^{2J}}   \bigg| \log \bigg| \sum_{j \leq J} b_j e^{u_j + iv_j } \bigg|  \bigg|^{2k} e^{ -    \sum_j (u_j^2 + v_j^2 )/M} d \vecu d \vecv  \ll  M^{J+k } (C  k )^k + M^{J } (C k )^{2k}. $$

\subsection{Proof of Proposition \ref{prop3}}\label{sec proof prop3}
By symmetry, it is enough to estimate
 \begin{align*}
 \mathcal{E}_{\vecm, \vecn, 1 } (\theta, T)  := &   \int_{\mathbb{R}^{J}}   \int_{\mathcal{R}_1}  \log \bigg|  1 +  \sum_{j=2}^J \frac{b_j}{b_1}  e^{(u_j -u_1 +i (v_j -v_1)) \sqrt{ \theta \log \log T }  }  \bigg| e^{ -  \sum_j (u_j^2+v_j^2)/\xi_j }       \vecu^{\vecm} \vecv^{\vecn}    d\vecu d\vecv.
 \end{align*}
Let $A_T =  ( \log \log \log T)/4 $. We divide $\mathcal{R}_1 $ into a disjoint union of the sets:
\begin{align*}
 \mathcal{R}_{1,S} := \{  (u_1 , \dots, u_J ) \in \mathcal{R}_1 :  &  - \frac{A_T }{ \sqrt{  \theta \log \log T  }}  < u_\ell - u_1 \leq 0 ~~ \mathrm{for} ~~ \ell \in  S, \\
&   u_j - u_1 \leq - \frac{A_T }{  \sqrt{ \theta \log \log T }} ~~ \mathrm{for} ~~ j \in \{ 2, \dots, J\} \setminus S \}   
\end{align*}
for $S \subset  \{ 2, \dots, J \}$. Let
\begin{align*}
 \mathcal{E}_S := & \mathcal{E}_{\vecm, \vecn, 1, S}(\theta, T) \\
    := &  \int_{\mathbb{R}^{J}}   \int_{\mathcal{R}_{1,S}}  \log \bigg|  1 +  \sum_{j=2}^J \frac{b_j}{b_1}  e^{(u_j -u_1 +i (v_j -v_1)) \sqrt{ \theta \log \log T }  }  \bigg| e^{ -  \sum_j (u_j^2+v_j^2)/ \xi_j }       \vecu^{\vecm} \vecv^{\vecn}    d\vecu d\vecv ,
    \end{align*}
so that 
$$ \mathcal{E}_{\vecm,\vecn,1} ( \theta, T) = \sum_{ S \subset \{ 2, \dots , J\} } \mathcal{E}_S.$$
First consider $ \mathcal{E}_\emptyset $. In this case it is easy to see that
$$  \mathcal{E}_\emptyset  =       \int_{\mathbb{R}^{J}} \int_{\mathcal{R}_{1,\emptyset } }  O(  e^{-A_T } )  e^{ -  \sum_j (u_j^2+v_j^2)/ \xi_j }     \vecu^{\vecm} \vecv^{\vecn}  d\vecu  d\vecv = O(  e^{-A_T } ) .  $$
Next consider $ S \neq \emptyset$, then there is at least one element $ \ell \in S$. We first observe the $u_\ell $ integral:
\begin{align*}
& \int_{ u_1 - \frac{A_T }{ \sqrt{ \theta \log \log T}}}^{u_1}     \log \bigg|  1 +  \sum_{j=2}^J \frac{b_j}{b_1}  e^{(u_j -u_1  +i v_j   - iv_1 )\sqrt{ \theta \log \log T}} \bigg| e^{ -   u_\ell^2/\xi_\ell  } u_\ell^{m_\ell} du_\ell \\
& \ll  \int_{  - \frac{A_T }{ \sqrt{ \theta \log \log T}}}^{0}   \bigg|   \log \bigg| \frac{b_\ell}{ b_1} e^{(u_\ell +i v_\ell   - iv_1 )\sqrt{ \theta \log \log T}} +    \sum_{ j \neq  \ell } \frac{b_j}{b_1}  e^{(u_j -u_1  +i v_j   - iv_1 )\sqrt{ \theta \log \log T}} \bigg| \bigg| du_\ell \\
& \ll   \frac{1}{ \sqrt{ \log \log T}}  \int_{ e^{-A _T} }^{1}    \bigg|   \log \bigg| \frac{b_\ell}{ b_1} w  e^{i( v_\ell  - v_1 )\sqrt{ \theta\log \log T}} +     \sum_{ j \neq \ell } \frac{b_j}{b_1}  e^{(u_j -u_1  +i v_j   - iv_1 )\sqrt{\theta \log \log T}} \bigg| \bigg|    \frac{dw}{w} \\
& \ll  \frac{A_T }{ \sqrt{  \log \log  T}}   +  \frac{1}{ \sqrt{   \log \log  T}}  \int_{ e^{-A_T} }^{1}   \bigg|    \log \bigg|  w  +    \sum_{ j \neq \ell } \frac{b_j}{b_\ell }  e^{(u_j -u_1   +i v_j   - iv_\ell  )\sqrt{ \theta\log \log T}} \bigg| \bigg|  \frac{d w}{w }    \\
& \ll    \frac{  e^{A_T}}{ \sqrt{ \log \log  T}}
\end{align*}
by the substitution $ w = e^{u_\ell \sqrt{ \theta \log \log T}}$. Here, the last inequality holds by the following lemma. 
\begin{lemma}\label{lem:bound log int}
Let $B$ be a fixed positive real number and let $ \epsilon_T >0 $ be a decreasing function to $0$ as $T \to \infty$. Then we have
$$    \int_{\epsilon_T}^1 | \log | u+z| | \frac{du}{u} = O\bigg(  \frac{1}{ \epsilon_T} \bigg) $$
as $T \to \infty$ uniformly for all $|z| \leq B$. 
\end{lemma} 
\begin{proof}
We first observe that it is enough to prove that
$$   \int_0^1 | \log | u+z| | du = O(1)$$
uniformly for bounded $z = \alpha + i \beta$.   By the inequality
\begin{align*}
   - |\log |u+\alpha|| \leq & \log |u+\alpha |  \leq  \log | u+z| = \log \sqrt{  (u+\alpha)^2 + \beta^2   } \leq \log \sqrt{  2 \max \{ (u+\alpha)^2 , \beta^2 \} } \\
  &  = \log \sqrt{2} + \max \{ \log |u+\alpha|, \log |\beta| \} \leq  \log \sqrt{2} +  | \log |u+\alpha| | +    \log B ,    
  \end{align*}
we see that
 $$   \int_0^1 | \log | u+z| | du  \leq    \int_0^1 | \log | u+\alpha | | du   + O(1)  . $$
If $ |\alpha | \geq 2 $, then it is easy to see that 
$$\int_0^1 | \log | u+\alpha | | du = \int_0^1 | \log | \alpha | + O(1)  | du = O(1) . $$
If $ |\alpha | < 2 $, then we split the interval into two intervals depending on the condition $ \log | u+\alpha | \geq 0 $. Thus,
$$ \int_0^1 | \log |u+\alpha || du =   \int_{ |u+\alpha| \geq 1 }      \log |u+\alpha |  du +  \int_{ |u+\alpha| <  1 }     -  \log |u+\alpha |  du . $$
It is easy to see that 
$$  0 \leq  \int_{ [0,1] \cap \{ |u+\alpha| \geq 1 \} }      \log |u+\alpha |  du  \leq \log (1+ B) $$
and 
$$  0 \leq  \int_{ [0,1] \cap \{ |u+\alpha| <  1\} }     -  \log |u+\alpha |  du  \leq 2 \int_0^1 - \log u du  \leq 2 . $$
\end{proof}

Hence, we find that
$$ \mathcal{E}_S  = O \bigg(   \frac{  e^{A_T}}{ \sqrt{ \log \log T}} \bigg) $$
for $ S \neq \emptyset$ and  
$$ \mathcal{E}_{\vecm, \vecn, 1 } (\theta, T) = O(e^{-A_T }) + O\bigg(  \frac{ e^{A_T}}{ \sqrt{ \log \log  T}}   \bigg) =  O\bigg(  \frac{ 1}{  ( \log \log T)^{1/4}}    \bigg) .$$

\subsection{Proof of Lemma \ref{lemma zero density}}\label{sec proof zero density lemma}

 We see that for a fixed real $\beta $ and for each $ i = 1,2 $ 
 $$ \theta_1^\beta - \theta_2^\beta = H_T  \big( \beta \theta_i^{\beta -1} + O(H_T  ) \big) .$$
 Thus, by \eqref{eqn:main int asymp} 
\begin{equation*} \begin{split}
 \theta_1^\alpha    I_{  \vecm, \vecn}( \theta_1 , T) & - \theta_2^\alpha  I_{  \vecm, \vecn}( \theta_2 , T)  \\
= &  H_T   \sqrt{    \log \log T} d_{\vecn}      \bigg( \alpha+ \frac12 \bigg)  \theta_i^{   \alpha  -1/2   }    \sum_{ \ell=1}^J   \int_{\mathcal{R}_\ell }     u_\ell         e^{ -  \sum_j  u_j^2 /    \xi_j  }     \vecu^{\vecm}     d\vecu  \\
& + H_T    d_{\vecn}  \alpha  \theta_i^{   \alpha    -1 }     \sum_{ \ell=1}^J   \int_{\mathcal{R}_\ell }     \log |b_\ell |       e^{ -  \sum_j  u_j^2 /   \xi_j  }        \vecu^{\vecm}     d\vecu \\
&   + \theta_1^\alpha \mathcal{E}_{\vecm, \vecn } (\theta_1, T) -  \theta_2^\alpha   \mathcal{E}_{\vecm, \vecn } (\theta_2, T)    + O( H_T  ^2 \sqrt{ \log \log T})  
 \end{split}\end{equation*}
 for each $ i = 1,2 $. Recall that $  \mathcal{E}_{\vecm, \vecn } (\theta, T)    := \sum_{\ell=1}^J  \mathcal{E}_{\vecm, \vecn, \ell } (\theta, T)   $. Hence, without loss of generality, we consider  
 \begin{align*}
 \theta_1^\alpha &\mathcal{E}_{\vecm, \vecn,1 } (\theta_1, T) -  \theta_2^\alpha  \mathcal{E}_{\vecm, \vecn,1 } (\theta_2, T)   \\
   = &   \int_{\mathbb{R}^{J}}   \int_{\mathcal{R}_1}  \log \bigg|  1 +  \sum_{j=2}^J \frac{b_j}{b_1}  e^{(u_j -u_1 +i (v_j -v_1)) \sqrt{  \log \log T }  }  \bigg|  \\
   & \times   \bigg(  \theta_1^{\alpha'}  e^{ -  \sum_j (u_j^2+v_j^2)/( \theta_1 \xi_j) }       -  \theta_2^{\alpha'}    e^{ -  \sum_j (u_j^2+v_j^2)/( \theta_2 \xi_j) }       \bigg)  \vecu^{\vecm}\vecv^{\vecn}    d\vecu d\vecv,
 \end{align*}
 where
 $$ \alpha'= \alpha -J - \frac{ \mathcal{K}(\vecm+\vecn)}{2}.  $$
 We see that 
\begin{align*}
 \theta_1^{ \alpha'}   e^{ -  \sum_j (u_j^2+v_j^2)/( \theta_1 \xi_j) }   &    -  \theta_2^{ \alpha'}     e^{ -  \sum_j (u_j^2+v_j^2)/( \theta_2 \xi_j) }  \\
 = &  \int_{\theta_2}^{\theta_1}     \frac{ \partial}{\partial w}\bigg( w^{ \alpha'}   e^{ -  \sum_j (u_j^2+v_j^2)/( w \xi_j) }\bigg)   dw\\
 = &  \int_{\theta_2}^{\theta_1}    \bigg(   \sum_{j \leq J}  \frac{ u_j^2 + v_j^2}{ w^2 \xi_j }  + \frac{ \alpha'}{w}      \bigg) w^{ \alpha'}  e^{ -  \sum_j (u_j^2+v_j^2)/( w \xi_j) }   dw \\
  \ll  & H_T      \bigg(   \sum_{j \leq J} ( u_j^2 + v_j^2)   +  1      \bigg) e^{ -  \sum_j (u_j^2+v_j^2)/( \theta_2  \xi_j) } .  
\end{align*} 
Thus, by adapting the proof of Proposition \ref{prop3} we find that
 \begin{align*}
 \theta_1^\alpha \mathcal{E}_{\vecm, \vecn,1 }& (\theta_1, T) -  \theta_2^\alpha  \mathcal{E}_{\vecm, \vecn,1 } (\theta_2, T)   \\
   \ll  &  H_T   \int_{\mathbb{R}^{J}}   \int_{\mathcal{R}_1} \bigg|  \log \bigg|  1 +  \sum_{j=2}^J \frac{b_j}{b_1}  e^{(u_j -u_1 +i (v_j -v_1)) \sqrt{  \log \log T }  }  \bigg|   \bigg|  e^{ -  \sum_j (u_j^2+v_j^2)/( \theta_2  \xi_j) }   \\
   &  \bigg(   \sum_{j \leq J} ( u_j^2 + v_j^2)   +  1      \bigg)  \big| \vecu^{\vecm } \vecv^{\vecn}\big|  d\vecu d\vecv \\
   \ll&  \frac{H_T   }{ ( \log \log T)^{1/4}} 
 \end{align*}
 and this completes the proof of the lemma.

\section{acknowledgement}

This work was partly done while the author was visiting Tokyo Institute of Techology in December 2017 and Nagoya University in February 2018. He would like to thank Masatoshi Suzuki and Kohji Matsumoto for their support and useful comments to this project. This work has been supported by the National Research Foundation of Korea (NRF) grant funded by the Korea government(MSIP) (No. 2016R1C1B1008405).

\end{document}